\documentclass{amsart}

\usepackage{amsmath,amssymb,latexsym,amsthm,graphicx}

\newtheorem{theorem}{Theorem}[section]

\newtheorem{lemma}[theorem]{Lemma}

\newtheorem{prop}[theorem]{Proposition}

\newtheorem{rem}[theorem]{Remark}

\newcommand{\mA}{\mathcal{A}}

\newcommand{\mR}{\mathcal{R}}

\newcommand{\mD}{\mathcal{D}}

\newcommand{\mQ}{\mathcal{Q}}

\newcommand{\mB}{\mathcal{B}}

\newcommand{\rN}{\mathbb{R}}

\newcommand{\Hs}{\mathbb{H}}

\newcommand{\nN}{\mathbb{N}}

\newcommand{\uS}{\mathbb{S}}

\newcommand{\eg}{\varepsilon}

\newcommand{\ag}{\alpha}

\newcommand{\sg}{\sigma}

\title{Spherical mean transform from the pde point of view}

\author{Linh V. Nguyen}
\address{Department of Mathematics, University of Idaho, Moscow, Idaho 83843}

\begin{document}

\maketitle

\begin{abstract} We study the spherical mean transform on $\rN^n$. The transform is characterized by the Euler-Poisson-Darboux equation. By looking at the spherical harmonic expansions, we obtain a system of $1+1$-dimension hyperbolic equations, which provide a good machinery to attack problems of spherical mean transform. 

As showcases, we discuss two known problems. The first one is a local uniqueness problem investigated by M.~Agranovsky and P.~Kuchment, [{\em Memoirs on Differential Equations and Mathematical Physics}, 52:1--16, 2011]. We present a simple proof which works even under a weaker condition. The second problem is to characterize the kernel of spherical mean transform on annular regions, which was studied by C. Epstein and B. Kleiner [{\em Comm. Pure Appl. Math.}, 46(3):441--451, 1993]. We present a short proof that simultaneously obtains the necessity and sufficiency for the characterization. As a consequence, we derive a reconstruction procedure for the transform with additional interior (or exterior) information. 

We also discuss how the approach works for the hyperbolic and spherical spaces.

\end{abstract}


\section{Introduction}\label{S:Intro}
Let $f$ be a function defined on $\rN^n$. The spherical mean transform $\mR(f)$ of $f$ is defined as $$\mR(f)(x,t) = \frac{1}{|S(x,t)|} \int\limits_{S(x,t)} f(x) d\sg(x). $$ Here, $S(x,t) \subset \rN^n$ is the sphere of radius $t$ centered at $x$, $d\sg(x)$ is the surface measure on $S(x,t)$, and $|S(x,t)|$ is the total measure of $S(x,t)$.  The spherical mean transform has been intensively studied due to its applications to PDEs, approximation theory, inverse scattering, and thermoacoustic tomography (e.g., \cite{CH2,Jo,Be1,Be2,FHR,LP1,LP2,AQ}).

The spherical mean transform can be characterized by the Darboux (or Euler-Poisson-Darboux) equation. Namely, if $f$ is smooth enough (for example $f \in C^\infty(\rN^n))$, then $G(x,t)=\mR(f)(x,t)$ satisfies (e.g., \cite{Jo}):
\begin{equation}\label{E:D} \left\{\begin{array} {l} G_{tt}(x,t) + \frac{n-1}{t} G_t(x,t) -\Delta_x G(x,t)=0,~  (x,t) \in \rN^n \times \rN_+,\\
G(x,0)=f(x),~ G_t(x,0)=0, \quad x \in \rN^n. \end{array}\right. \end{equation}
Conversely, if $G(x,t) \in C^\infty(\rN^n \times \overline{\rN}_+)$ satisfies the above equation then $G=\mR(f)$.
Let us recall the polar coordinate decomposition of the Laplacian in $\rN^n$: $$\Delta = \frac{\partial^2}{\partial r^2} + \frac{n-1}{r} \frac{\partial}{\partial r} + \frac{1}{r^2} \Delta_{\uS^{n-1}},$$ where $r= |x|$ and $\Delta_{\uS^{n-1}}$ is the Laplace-Beltrami operator on the unit sphere $\uS^{n-1}$.  Hence, for any twice differentiable function $g=g(r)$: $$\Delta\left[ g(r) r^m Y^m(\theta) \right] = r^m Y^m(\theta) \left [\frac{\partial^2}{\partial r^2} + \frac{n-1+2m}{r}\frac{\partial}{\partial r} \right] g(r),$$
where $Y^m$ is any spherical harmonics of degree $m$.

Let us expand $f$ and $G$ in terms of spherical harmonics: \begin{eqnarray}  \label{E:Decom} f(x) = \sum_{m=0}^\infty \sum_{l=1}^{l_m} f_{ml}(r) r^m Y^m_{l}(\theta),~ G(x, t) =\sum_{k=0}^\infty \sum_{l=1}^{l_m} g_{ml}(r,t) r^m Y^m_{l}(\theta).\end{eqnarray}
For each $m \geq 0$, we denote $\mB_{r,m} = \left[ \frac{\partial ^2}{\partial r^2} + \frac{n-1+2m}{r} \frac{\partial}{\partial r} \right]$. Then, (\ref{E:D}) is equivalent to \begin{equation}\label{E:1dD} \left\{\begin{array} {l}\mB_{t,0} g_{ml}(r,t)  - \mB_{r,m}g_{ml} (r,t) =0,~ (r,t) \in \rN \times \rN_+, \\
g_{ml} (r,0)=f_{ml}(r),~\partial_t g_{ml} (r,0)=0,~r \in \rN, \end{array}\right. \end{equation} for each $m \geq 0$ and $l=1,..,l_m$. This equation can be symmetrized by a differential operator. Indeed, let 
$$\mQ_m = \left(\frac{r}{n} \frac{\partial}{\partial r} +1 \right)\left(\frac{r}{n+2} \frac{\partial}{\partial r} +1 \right)...\left(\frac{r}{n+2(m-1)} \frac{\partial}{\partial r} +1 \right).$$
We then have (see Appendix):
\begin{equation} \label{E:Qk} \mQ_m \mB_{r,m} = \mB_{r,0} \mQ_m.\end{equation}
Applying $\mQ_m$ to (\ref{E:1dD}), we obtain the symmetric equation \begin{equation} \label{E:1dDs}\left[\mB_{t,0} -\mB_{r,0}\right]\ag (r,t) =0,\end{equation} where  $\ag = \mQ_m g_{ml}$, for any $m \in \nN$ and $l=1,..,l_m$.

In this article, we exploit equations (\ref{E:1dD}) and (\ref{E:1dDs}) to present some short and clear proofs for some (old and new) properties of spherical mean transform.

The article is organized as follows. In Section \ref{S:Local}, we revisit a local uniqueness result by Agranovsky and Kuchment \cite{AK-S}. We also point out that the same result still holds under a weaker condition.  We then discuss a result by Epstein and Kleiner \cite{Ep}  for spherical mean transform on the annular region in Section {\ref{S:Ann}}. We provide a reconstruction procedure for the transform when some additional interior (or exterior) information is provided. Finally, we discuss how the decompostion and symmetrization work on the hyperbolic and spherical spaces.

\section{A local uniqueness result}\label{S:Local}
Let us recall the following result by Agranovsky and Kuchment \cite[Theorem 10]{AK-S}: 
\begin{theorem}\label{T:Main} Let $f \in C^\infty(B_{R+\eg})$ such that $f(x)=0$ for $x \in B_R$. Assume that $\mR(f)(x,R)=0$ for all $x \in B_\eg$. Then, $f(x) = 0$ for $x \in B_{R+\eg}$. Here, $B_R$ is the sphere of radius $R$ centered at $0$. \end{theorem}
The above theorem was proved by utilizing the notion of ridge function and Titchmarsh theorem. We present here an alternative proof using equation (\ref{E:1dD}) and simple energy arguments. Let us start with a basic domain of dependence argument which will be repeatedly used in this article: 
\begin{prop}\label{P:DoD} Let $c>0$, $\eg \geq 0$, and $u$ satisfy the equation: \begin{eqnarray}\label{E:DG} \left\{ \begin{array}{l} \left(\frac{\partial ^2}{\partial r^2} + \frac{c}{r} \frac{\partial}{\partial r} -\Delta_y \right) u(r,y) =0,~y \in \rN^n,~ r \geq \eg, \\ u(\eg,y)=u_\eg(y),~u_r(\eg,y)=0,~ y \in \rN^n. \end{array}  \right.\end{eqnarray}
Assume that $u_\eg(y)=0$ for all $y \in B(y_0,r_0)$. Then, $u(r,y) =0$ if $r \geq \eg$ and $|y-y_0|+r \leq r_0+\eg$. \end{prop}
\noindent The proof Proposition \ref{P:DoD} is well known (see, e.g., \cite{CH2}). It consists in only a simple energy argument. We will provide in Appendix for the sake of completeness.
\begin{proof}[Proof of Theorem \ref{T:Main}]
Let us consider the series expansion (\ref{E:Decom}) of $f$ and $G=\mR(f)$. We observe that the conditions $f(x)=0$ for $ x \in B_R$ and $G(x,R) = \mR(f)(x,R) =0$ for $x \in B_\eg$ imply \begin{equation}\label{E:Conds} g_{ml}(0,t)=0,~ t \in [0,R], \mbox{ and }  g_{ml}(r,R)=0,~s \in [0,\eg]. \end{equation}
The goal is to prove that $f_{ml}(t) = g_{ml}(0,t) = 0$ for $t \in [0,R+\eg]$. Our idea is to transform equation (\ref{E:1dD}): 
 \begin{eqnarray*} \left\{\begin{array} {l} \left[ \frac{\partial ^2}{\partial t^2} + \frac{n-1}{t} \frac{\partial}{\partial t} \right] g_m(r,t)  -\left[ \frac{\partial ^2}{\partial r^2} + \frac{n-1+2m}{r} \frac{\partial}{\partial r} \right] g_m(r,t) =0,~ (r,t) \in \rN^2, \\
g_{ml} (r,0)=f_{ml}(r),\quad \partial_t g_{ml} (r,0)=0,~\forall r \in \rN, \end{array}\right. \end{eqnarray*}
 to a Darboux equation for which $r$ is the temporal variable.  Indeed, let us introduce the function $u(r,y) =g_{ml}(r, |y|)$, which is radially symmetric with respect to the variable $y \in \rN^n$. Then for any $y$ such that $|y|=t$,  $$ \Delta_y u(r,y) = \left[ \frac{\partial ^2}{\partial t^2} + \frac{n-1}{t} \frac{\partial}{\partial t} \right] g_m(r,t). $$
We, hence, obtain the Darboux equation: \begin{equation} \label{E:Std}\left[ \frac{\partial ^2}{\partial r^2} + \frac{n-1+2m}{r} \frac{\partial}{\partial r} \right] u(r,y) - \Delta_y u(r,y) =0.\end{equation}
From (\ref{E:Conds}), we obtain the following zero initial and boundary values of $u$:  $$u(0,y) =0,~\forall y \in B_R, \quad u(r,y)=0,~\forall (r,y) \in [0,\eg] \times S_R.$$
A simple energy argument shows that $u(r,y)=0$ for all $(r,y) \in B_R \times [0,\eg]$. Indeed, let $$E(r) = \int\limits_{B(0,R)} \left[|u_r(r,y)|^2 +|\nabla u(r,y)|^2 \right] dy.$$ Since $u(r,y)=0$ for $y \in S_R$, integration by parts gives $$\frac{d E(r)}{dr} =2  \int\limits_{B(0,R)} \left[u_r(r,y) -\Delta u(r,y) \right] u_r(r,y) dy.$$ Due to (\ref{E:Std}), $$\frac{d E(r)}{dr} =- \frac{2(n-1)}{r} \int\limits_{B(0,R)} |u_r(r,y)|^2 dy \leq 0.$$
We obtain $E(r) \leq E(0)=0$ for all $0 \leq r \leq \eg$. This implies $E(r)=0$ or $u(r,y)=0$ for all $(r,y) \in [0,\eg] \times B_R$.

We now arrive to the equation \begin{eqnarray*}\left\{\begin{array}{l} \left(\frac{\partial ^2}{\partial r^2} + \frac{n-1+2m}{r} \frac{\partial}{\partial r} -\Delta_y \right) u(r,y) =0, \\ u(\eg,y) = u_r(\eg,y)=0,~ y \in \rN^n.\end{array} \right. \end{eqnarray*}
Due to Proposition \ref{P:DoD}, we obtain $u(y,t)=0$ for $|y|+r \leq R + \eg$. This means $g_{ml}(r,t)=0$ for all $(r,t)$ such that $r+t \leq R+ \eg$.  In particular, $f_{ml}(s) = g_{ml}(0,s)=0$ for $s \in [0,R+\eg]$.

We conclude, due to the expansion (\ref{E:Decom}),  $f(x) =0$ for all $x \in B_{R+\eg}$. \end{proof}

The condition $f(x)=0$ for $x \in B_R$ is equivalent to $\mR(f)(x,t)=0$ for all $(x,t)$ such that $|x|+t \leq R$. It is much stronger than: $D^\ag_x\mR(f)(0,t)=0$ for any multiindex $\ag$ and $t \in [0,R]$. Following the above proof, we obtain the same conclusion under this weaker condition:
\begin{theorem}\label{T:Mainc} Let $f \in C^\infty(B_{R+\eg})$ such that $D_x^\alpha \mR(f)(0,t)=0$ for all multi-indices $\ag$ and $t \in [0,R]$. Assume that $R(f)(x,R)=0$ for all $x \in B_\eg$. Then $f (x)= 0$ for all $x \in B_{R+\eg}$. \end{theorem}

\section{Spherical mean transform on the annular region} \label{S:Ann}
Given $0<a<A$, let us consider the annular region $$Ann(a,A)=\{x \in \rN^n: a<|x|<A\}.$$ We denote by  $Z^\infty(a,A)$ the space of all functions $f \in C^\infty(Ann(a,A))$ satisfying $\mR(f)(x,t) =0$ for all $(x,t) \in \mA$. Here, \begin{eqnarray}\label{E:Annu} \mA = \{(x,t): S(x,t) \subset Ann(a,A) \mbox{ and } S(0,a) \subset B(x,t) \},\end{eqnarray} where $B(x,t)$ is the ball of radius $t$ centered at $x$. The following result characterize all the functions $f\in Z^\infty(a,A)$:
\begin{theorem}\label{T:Ep} A function $f$ belongs to $ Z^\infty(a,A)$ if and only if $f_0=0$ and $f_{ml}$ is of the form  \begin{eqnarray}\label{E:fm} f_{ml}(r) = \sum_{i=0}^{k-1} c^i_{ml} r^{m-(n+2i)}, \end{eqnarray} for any $m>0$ and $\l=1,..,l_m$.
\end{theorem}
This result was proved in \cite{Ep}. The necessity was obtained by a projection formula on spaces of spherical harmonics. The sufficiency was proved by a series of lemmas concerning the connection between the functions in $Z^\infty(a,\infty)$ and harmonic functions with certain behavior at infinity.  In this section, we present an argument to simultaneously prove the necessity and sufficiency. It also provides a reconstruction procedure for the transform with additional interior/exterior information, which might find applications in biomedical imaging.

\begin{rem} By some convolution arguments as in \cite{Ep}, we can deduce from Theorem \ref{T:Ep} the same characterization for functions $f \in Z(a,A)$. Here, $Z(a,A)$ is the set of functions $f \in C(Ann(a,A))$ such that $\mR(f)(x,t)=0,~\forall (x,t) \in \mA$.  \end{rem}

Without loss of generality, we might also assume that $f \in C^\infty(\overline{Ann(a,A)})$. Otherwise, we prove the same characterization on the region $Ann(a+\eg,A-\eg)$ (for small enough $\eg>0$) and then let $\eg \to 0$.

Let us extend $f$ smoothly to $\rN^n$. Then, $G=\mR(f) \in C^\infty(\rN^n \times \overline{\rN}_+)$ satisfies the equation (\ref{E:D}). The condition $f \in Z^\infty(a,A)$ is equivalent to \begin{equation}\label{E:Gvn} G(x,t)=0,~(x,t) \in \rN^n \times (a,A), \mbox{ such that } |x| + \left|t -\frac{a+A}{2} \right| \leq \frac{A-a}{2}. \end{equation}
\begin{lemma}\label{L:ZAa} Consider the decomposition of $G$ as in (\ref{E:Decom}). Then $f \in Z^\infty(a,A)$ if and only if \begin{equation}\label{E:ag} g_{ml}(0,t)=0, ~\forall t \in (a,A), \end{equation} for all $m \in \nN$ and $l=1,..,l_m$. \end{lemma}
\begin{proof} It suffices to prove that (\ref{E:Gvn}) and (\ref{E:ag}) are equivalent. The fact that (\ref{E:Gvn}) implies (\ref{E:ag}) is obvious. We now prove the other implication. Similar to the proof of Theorem \ref{T:Main}, we translate equation (\ref{E:1dD}) to a Darboux equation by introducing the function $u(r,y) =g_{ml}(r, |y|)$. We then obtain: \begin{eqnarray*}  \left(\frac{\partial ^2}{\partial r^2} + \frac{n-1+2m}{r} \frac{\partial}{\partial r} -\Delta_y \right) u(r,y) =0. \end{eqnarray*}
 Let $y_0 \in \rN^n$ such that $|y_0|=\frac{a+A}{2}$. The condition $g_{ml}(0,t)=0$ for $t \in (a,A)$ implies $$u(0,y)=0,~\mbox{for all }y \in B(y_0,r_0),$$ where $r_0=\frac{A-a}{2}$. The domain of dependence argument (see Proposition \ref{P:DoD}) then implies $$u(r,y)=0,~\mbox{ for all $(y,r) \in \rN^n \times \rN_+$ such that } |y-y_0|+ r \leq r_0.$$
Recalling that $u(r,y) = g_{ml}(r,|y|)$, we obtain $$g_{ml}(r,t)=0, \mbox{ for all } (x,t) \mbox{ such that } \left|t - |y_0| \right| +r \leq r_0. $$
That is, $$g_{ml}(r,t)=0, \mbox{ for all } (r,t) \mbox{ such that } \left|t - \frac{a+A}{2}\right| +r \leq \frac{A-a}{2}. $$
Since this is true for all $m \in \nN$ and $l=1,..,l_m$, we conclude $$G(x,t)=0, \mbox{ for all } (x,t) \mbox{ such that } \left|t - \frac{a+A}{2}\right| +|x| \leq \frac{A-a}{2}. $$
This finishes our proof. \end{proof}
We now present our proof for Theorem \ref{T:Ep}. It simultaneously provides the necessity and sufficiency of the characterization in Theorem \ref{T:Ep}.
\begin{proof}[Proof of Theorem \ref{T:Ep}] Let $\ag(r,t) = \mQ_m g_{ml}(r,t)$, where \begin{equation}\label{E:Qm}\mQ_m = \left(\frac{r}{n} \frac{\partial}{\partial r} +1 \right)\left(\frac{r}{n+2} \frac{\partial}{\partial r} +1 \right)...\left(\frac{r}{n+2(m-1)} \frac{\partial}{\partial r} +1 \right).\end{equation}
Due to (\ref{E:Qk}), we arrive to the symmetric equation (\ref{E:1dDs}): 
 \begin{eqnarray*} \left[ \frac{\partial ^2}{\partial t^2} + \frac{n-1}{t} \frac{\partial}{\partial t} \right] \ag(r,t)  -\left[ \frac{\partial ^2}{\partial r^2} + \frac{n-1}{r} \frac{\partial}{\partial r} \right] \ag(r,t) =0,~ (r,t) \in \rN^2\end{eqnarray*}
We obtain (e.g., \cite{Jo,HelGeo}):
$$\ag(s,0) = \ag (0,s).$$ Simple observations show: $$\ag(0,s) = g_{ml}(0,s), ~\ag(0,s) = [\mQ_m f_{ml}](s).$$ 
Therefore, \begin{eqnarray} \label{E:fg} [\mQ_{m}f_{ml}](s) = g_{ml}(0,s). \end{eqnarray}
From Lemma \ref{L:ZAa}, $f \in Z^\infty(a,A)$ if and only if $g_{ml}(0,s)=0$ for $s \in (a,A)$. Or, equivalently,  \begin{eqnarray}\label{E:Qm0} [\mQ_m f_{ml}](s)=0,~\forall s \in (a,A).\end{eqnarray}
Due to formula (\ref{E:Qm}) of $\mQ_m$, (\ref{E:Qm0}) is equivalent to \begin{eqnarray*}  f_{ml}(r) = \left\{ \begin{array}{l} \sum\limits_{i=0}^{m-1}c^i_{ml} r^{-n-2i}, \mbox{ if } m \geq 1, \\ 0,\quad \mbox { if } m=0.\end{array} \right. \end{eqnarray*}
This finishes our proof. \end{proof}

We are now interested in reconstructing $f$ on $Ann(a,A)$ from $\mR(f)|_\mA$. The problem does not have a unique answer since the kernel described in Theorem \ref{T:Ep} is nontrivial.  We, therefore, make one more assumption: $f \in C^\infty(Ann(r_0, A))$ and it is known in the interior region $Ann(r_0,a]=\{x \in \rN^n: r_0 < |x| \leq a\}$, for some number $0 \leq r_0<a$. The problem of reconstructing $f$ from $\mR(f)|_{\mA}$ and $f|_{Ann(r_0,a]}$ resembles interior tomography with prior interior knowledge, which is widely investigated in biomedical imaging (e.g.,\cite{kudo2008tiny,courdurier2008solving}). The following result provide a formula useful for the reconstruction:

\begin{theorem} \label{T:Interior} Consider the spherical harmonic decomposition (\ref{E:Decom}) of $f$ and $G=\mR(f)$. Let $$k_{ml}(r) = \frac{n (n+2)...[n+2(m-1)]}{2^{(m-1)} (m-1)!}r^{-[n+2(m-1)]}\int\limits_a^r g_{ml}(0,\tau) \tau^{n-1}(r^2-\tau^2)^{m-1} d\tau.$$ 
Then for $r \in \rN_+$: \begin{eqnarray}\label{E:Diff} q_{ml}(r):= f_{ml}(r)-k_{ml}(r) = \left\{ \begin{array}{l} \sum\limits_{i=0}^{m-1}c^i_{ml} r^{-n-2i}, \mbox{ if } m \geq 1, \\ 0,\quad \mbox { if } m=0.\end{array} \right. \end{eqnarray}
\end{theorem}

\begin{proof}
We claim that \begin{eqnarray} \label{E:km} \mQ_{m} k_{ml} = g_{ml}.\end{eqnarray} Indeed, direct calculations show:
\begin{eqnarray*}\left(\frac{r}{n+2(m-1)} \frac{\partial}{\partial r} +1 \right) k_{ml}(r) = \frac{n ...[n+2(m-2)]}{2^{m-2}(m-2)!}r^{-[n+2(m-2)]} \\ \times  \int\limits_a^r g_{ml}(0,\tau) \tau^{n-1}(r^2-\tau^2)^{m-2} d\tau. \end{eqnarray*}
We recall that $$\mQ_m =\left(\frac{r}{n} \frac{\partial}{\partial r} +1 \right)\left(\frac{r}{n+2} \frac{\partial}{\partial r} +1 \right)...\left(\frac{r}{n+2(m-1)} \frac{\partial}{\partial r} +1 \right).$$
By induction, we obtain
\begin{equation*} \mQ_m k_{lm}(r) =\left(\frac{r}{n} \frac{\partial}{\partial r} +1 \right) \left[ nr^{-n} \int\limits_a^r g_{ml}(0, \tau) \tau^{n-1}d\tau\right]=g_{ml}(0, r).\end{equation*}
This proves (\ref{E:km}). Since $f_{ml}$ also satisfies the same equation (\ref{E:fg}), $$\mQ_m (f_{ml}-g_{ml}) =0.$$
Therefore, $q_{ml}(r):= f_{ml}(r)-k_{ml}(r)$ has the form (\ref{E:Diff}).
\end{proof}
We now arrive to a procedure to compute $f$ on $Ann(a,A)$ from $\mR{f}|_\mA$ and $f|_{Ann(r_0,a]}$:
\begin{itemize}
\item[1)] Compute $f_{ml}(r)$ for $r \in (r_0,a]$ from $f|_{Ann(r_0,a]}$.
\item[2)] Compute $g_{ml}(0,s)$ for $s \in (a,A)$ from $G=\mR(f)|_{\mA}$.
\item[3)] Compute $q_{ml}(s)=f_{ml}(s) - k_{ml}(s)$ and its derivatives at $s=a$ from the above knowledge of $f_{ml}$ and $g_{ml}$. Use them to compute $c^i_{ml}$.
\item[4)] Compute $f_{ml}= k_{ml} + q_{ml}$ on $(a,A)$. Then, compute $f$ on $Ann(a,A)$. 
\end{itemize}

\begin{rem} The same procedure also works when $f(x)$ is provided in an exterior domain $Ann[A,R)=\{A \leq |x| < R\}$, for some $R>A$. \end{rem}
\section{The approach for Hyperbolic and Spherical spaces}\label{S:DR}
We have considered the spherical mean transform from a PDE point of view. We decomposed the Euler-Poisson-Darboux equation into the $1+1$-dimension hyperbolic equations. By exploiting them and their symmetrized versions, we obtain some old and new properties for spherical mean transform. In this section, we describe the approach for the hyperbolic and spherical spaces.

Let us consider the spherical mean transform on the hyperbolic space $\Hs^n$, which is the unit ball with the metric $$ds^2= \frac{4}{(1-|x|^2)^2} dx^2.$$
The Laplace-Beltrami operator $\Delta$ on $\Hs^n$ is $\Delta = \frac{1}{4}(1-|x|^2)^{n} \nabla \left((1-|x|^2)^{2-n}\nabla \right)$. In terms of polar coordinates,\begin{equation} \label{E:Polar}\Delta = \frac{\partial^2}{\partial r^2}+(n-1) \coth(r) \frac{\partial}{\partial r} + \frac{1}{\sinh^2(r)} \Delta_{\uS^{n-1}}.\end{equation} Here, $r=d_{\Hs^n}(x,0)$ and $\Delta_{\uS^{n-1}}$ is the Laplace-Beltrami operator on the Euclidean unit sphere $\uS^{n-1} \subset \rN^n$ applying to the variable $\theta = \frac{x}{|x|}$.

The spherical mean transform $\mR(f)$ of a function $f$ is characterized by (e.g., \cite{HelGeo}):
\begin{eqnarray}\label{E:Darbh} \left\{ \begin{array}{l} [\frac{\partial^2}{\partial t^2} +(n-1)  \coth(t) \frac{\partial}{\partial t}  -\Delta] G(x,t) =0, (x,t) \in \Hs^n \times \rN_+.\\  G(x,0)=f(x),~ G_t(x,0) = 0,~ x \in \Hs^n. \end{array}\right. \end{eqnarray} Conversely, if $G(x,t) \in C^\infty(\Hs^n \times \overline{\rN}_+)$ satisfies the above equation,  then $G(x,t) = \mR(f)(x,t)$. Let us consider the decompositions\footnote{This is different from the decomposition (\ref{E:Decom}) on the Euclidean space.}: 
\begin{eqnarray}  \label{E:Decomh} f(x) = \sum_{m=0}^\infty \sum_{l=1}^{l_m} f_{ml}(r) Y^m_{l}(\theta),~ G(x, t) =\sum_{k=0}^\infty \sum_{l=1}^{l_m} g_{ml}(r,t) Y^m_{l}(\theta).\end{eqnarray}
Let us consider the operator $$\mD_{m,r} = \frac{\partial^2 }{\partial r^2} +(n-1) \coth(r) \frac{\partial}{\partial r} - \frac{m(m+n-2)}{\sinh^2 r}.$$ 
Since $\Delta_{\uS^{n-1}} Y^m_l=-m(m+n-2) Y_l^m$, for any function $g=g(r)$: $$\Delta [g(r) Y_l^m(\theta)] = (\mD_k g)(r) Y_l^m(\theta).$$
We obtain the following analog of equation (\ref{E:1dD}) in Section \ref{S:Intro}:
\begin{eqnarray} \label{E:1dDh} \left\{ \begin{array}{l} (\mD_{0,t} - \mD_{m,r}) g_{ml} (r,t)=0, \\  g_{ml} (r,0)=f_{ml} (r),~ \frac{\partial}{\partial r} g_{ml} (r,0)=0. \end{array} \right.
\end{eqnarray}
This equation can be symmetrized by the operator $\mQ_m= \Gamma_1...\Gamma_m$, where $\Gamma_{k} = \frac{\partial }{\partial r} + (n+k-2) \coth(r)$. Indeed, one can verify (see \cite{Constant}): $$\mQ_m \mD_{m,r} =\mD_{0,r} \mQ_m.$$ 
Due to (\ref{E:1dDh}), we obtain the symmetric equation for $\ag(r,t) = \mQ_m g_{ml}(r,t)$: \begin{eqnarray} \label{E:1dDhs}  (\mD_{0,t} - \mD_{0,r}) \ag(r,t)=0. \end{eqnarray}
This equation provides a good tool to work with spherical mean transform on $\Hs^n$. For example, it was exploited in \cite{Constant} to characterize the functions $f$ such that $\partial^\ag_x \mR(f)(0,R)=0$ for a fixed $R>0$ and all multi-index $\ag$.

For the spherical space $\uS^n$, which is the unit sphere in $\rN^{n+1}$, the same argument as above works; except that one need to replace the hyperbolic trigonometric functions ($\cosh, \sinh, \coth, etc$) by the usual ones ($\cos,\sin,\cot,etc$). 

Finally, we remark that one can use the same approach to investigate the ball transform (e.g., \cite{Volch-b}), instead of the spherical mean transform.

\section*{Appendix}
In this Appendix, we present the proofs of equation \ref{E:Qk} and Proposition \ref{P:DoD}.
\begin{proof}[{\bf Proof of equation \ref{E:Qk}}]
Let us recall that $$\mQ_m = \prod_{i=0}^{m-1} \left(\frac{r}{n+2i} \frac{\partial}{\partial r} +1 \right)= \left(\frac{r}{n} \frac{\partial}{\partial r} +1 \right)\left(\frac{r}{n+2} \frac{\partial}{\partial r} +1 \right)...\left(\frac{r}{n+2(m-1)} \frac{\partial}{\partial r} +1 \right).$$ One could easily check:
\begin{eqnarray*} && \left(\frac{r}{n+2(k-1)} \frac{\partial}{\partial r}  +1 \right)\mB_{r,k} =\mB_{r,k-1} \left(\frac{r}{n+2(k-1)} \frac{\partial}{\partial r}  +1 \right).\end{eqnarray*}
By induction, we obtain: $$\mQ_m \mB_{r,m} = \mB_{r,0} \mQ_m.$$ \end{proof}

\begin{proof}[\bf {Proof of Proposition \ref{P:DoD}}]
For $\eg < r \leq r_0+\eg$, let $$E(r) = \int\limits_{B(y_0,r_0+\eg -r)} [|u_r(r,y)|^2 + |\nabla u(r,y)|^2] dy.$$
Then \begin{eqnarray*} \frac{d E(r)}{dr} &=& 2 \int\limits_{B(y_0,r_0+\eg -r)} [u_r(r,y)u_{rr}(r,y) + \nabla u(r,y) \nabla u_r(r,y)] dy \\ &-&  \int\limits_{S(y_0,r_0+\eg -r)} (|u_r(r,y)|^2 + |\nabla u(r,y)|^2) dy.\end{eqnarray*}
Taking integration by parts, we obtain \begin{eqnarray*} \frac{d E(r)}{dr} &=& 2 \int\limits_{B(y_0,r_0+\eg -r)} [u_r(r,y)u_{rr}(r,y) - \Delta u(r,y) u_r(r,y)] dy \\ &-&  \int\limits_{S(y_0,r_0+\eg -r)} \left[|u_r(r,y)|^2 -2 u_r(r,y) \partial_\nu u(r,y) + |\nabla u(r,y)|^2\right] dy.\end{eqnarray*}
Here, $\partial_\nu$ is the outer normal derivative on $S(y_0,r_0+\eg-r)$. Since $|\partial_\nu u| \leq \|\nabla u\|$, the last integral is nonnegative. Therefore, 
\begin{eqnarray*} \frac{d E(r)}{dr}  && \leq  2 \int\limits_{B(y_0,r_0+\eg-r)} \left[u_r(r,y)u_{rr}(y,r) - \Delta u(r,y) u_r(r,y) \right] dy \\ && \leq 2 \int\limits_{B(y_0,r_0+\eg -r)} \left[u_{rr}(r,y) + \frac{n+2m-1}{r} u_r(r,y) - \Delta u(r,y) \right] u_r(r,y) dy.\end{eqnarray*}
Due to equation (\ref{E:DG}), we obtain $\frac{d E(r)}{dt} \leq 0$ for all $\eg < r \leq r_0+\eg$. Since $E(\eg) =0$, we obtain $E(r)=0$ for all $\eg \leq r \leq r_0+\eg$. This implies $u(r,y)=0$ for all $y \in B(y_0, r_0+\eg -r)$. Equivalently, $u(r,y)=0$ for all $(r,y)$ such that $r \geq \eg$ and $|y-y_0|+ r \leq r_0+\eg$.  \end{proof}


\begin{thebibliography}{KCND08}

\bibitem[AK11]{AK-S}
M.~Agranovsky and P.~Kuchment.
\newblock The support theorem for the single radius spherical mean transform.
\newblock {\em Memoirs on Differential Equations and Mathematical Physics},
  52:1--16, 2011.

\bibitem[AQ96]{AQ}
M.~Agranovsky and E.~T.~Quinto.
\newblock Injectivity sets for the {R}adon transform over circles and complete
  systems of radial functions.
\newblock {\em J. Funct. Anal.}, 139(2):383--414, 1996.

\bibitem[Bey83a]{Be1}
G.~Beylkin.
\newblock The fundamental identity for iterated spherical means and the
  inversion formula for diffraction tomography and inverse scattering.
\newblock {\em J. Math. Phys.}, 24(6):1399--1400, 1983.

\bibitem[Bey83b]{Be2}
G.~Beylkin.
\newblock Iterated spherical means in linearized inverse problems.
\newblock In {\em Conference on inverse scattering: theory and application
  ({T}ulsa, {O}kla., 1983)}, pages 112--117. SIAM, Philadelphia, PA, 1983.

\bibitem[CH62]{CH2}
R.~Courant and D.~Hilbert.
\newblock {\em Methods of mathematical physics. {V}ol. {II}: {P}artial
  differential equations}.
\newblock (Vol. II by R. Courant). Interscience Publishers (a division of John
  Wiley \& Sons), New York-Lon don, 1962.

\bibitem[CNDK08]{courdurier2008solving}
M.~Courdurier, F.~Noo, M.~Defrise, and H.~Kudo.
\newblock Solving the interior problem of computed tomography using a priori
  knowledge.
\newblock {\em Inverse problems}, 24:065001, 2008.

\bibitem[EK93]{Ep}
C.~Epstein and B.~Kleiner.
\newblock Spherical means in annular regions.
\newblock {\em Comm. Pure Appl. Math.}, 46(3):441--451, 1993.

\bibitem[FHR07]{FHR}
D.~Finch, M.~Haltmeier, and Rakesh.
\newblock Inversion of spherical means and the wave equation in even
  dimensions.
\newblock {\em SIAM J. Appl. Math.}, 68(2):392--412, 2007.

\bibitem[Hel84]{HelGeo}
S.~Helgason.
\newblock {\em Groups and geometric analysis}, volume 113 of {\em Pure and
  Applied Mathematics}.
\newblock Academic Press Inc., Orlando, FL, 1984.
\newblock Integral geometry, invariant differential operators, and spherical
  functions.

\bibitem[Joh81]{Jo}
F.~John.
\newblock {\em Plane waves and spherical means applied to partial differential
  equations}.
\newblock Springer-Verlag, New York, 1981.
\newblock Reprint of the 1955 original.

\bibitem[KCND08]{kudo2008tiny}
H.~Kudo, M.~Courdurier, F.~Noo, and M.~Defrise.
\newblock Tiny a priori knowledge solves the interior problem in computed
  tomography.
\newblock {\em Physics in medicine and biology}, 53:2207, 2008.

\bibitem[LP93]{LP1}
V.~Y. Lin and A.~Pinkus.
\newblock Fundamentality of ridge functions.
\newblock {\em J. Approx. Theory}, 75(3):295--311, 1993.

\bibitem[LP94]{LP2}
V.~Y. Lin and A.~Pinkus.
\newblock Approximation of multivariate functions.
\newblock In {\em Advances in computational mathematics ({N}ew {D}elhi, 1993)},
  volume~4 of {\em Ser. Approx. Decompos.}, pages 257--265. World Sci. Publ.,
  River Edge, NJ, 1994.

\bibitem[{Ngu}11]{Constant}
L.~V. {Nguyen}.
\newblock {Range description for a spherical mean transform on spaces of
  constant curvatures}.
\newblock {\em ArXiv E-print: 1107.1746}, July 2011.

\bibitem[Vol03]{Volch-b}
V.~V. Volchkov.
\newblock {\em Integral geometry and convolution equations}.
\newblock Kluwer Academic Publishers, Dordrecht, 2003.

\end{thebibliography}

\def\dbar{\leavevmode\hbox to 0pt{\hskip.2ex \accent"16\hss}d}

\end{document}